\documentclass{article}
\usepackage{PRIMEarxiv}
\usepackage{amsmath,amsfonts}
\usepackage{algorithmic}
\usepackage{algorithm}
\usepackage{array}
\usepackage[caption=false,font=normalsize,labelfont=sf,textfont=sf]{subfig}
\usepackage{textcomp}
\usepackage{stfloats}
\usepackage{url}
\usepackage{verbatim}
\usepackage{graphicx}
\usepackage[square,numbers]{natbib}
\usepackage{amsthm}
\usepackage{booktabs} 
\newtheorem{theorem}{Theorem}
\newtheorem{lemma}{Lemma}
\newtheorem{prop}{Proposition}

\def\pP{{\mathbb P}}
\def\eE{{\mathbb E}}
\def\rR{{\mathbb R}}

\title{On Learning the Distribution of a Random Spatial Field in a
	Location-Unaware Mobile Sensing Setup
}

\author{
	Meera Pai \\
	Department of Electrical Engineering \\
	IIT Bombay,
	India\\
	\texttt{meeravpai@ee.iitb.ac.in} \\
}

\begin{document}
	\maketitle

	\begin{abstract}
	 In applications like environment monitoring and pollution control, physical quantities are modeled by spatio-temporal fields. It is of interest to learn the statistical distribution of such fields as a function of space, time or both. In this work, our aim is to learn the statistical distribution of a spatio-temporal field along a fixed one dimensional path, as a function of spatial location, in the absence of location information. Spatial field analysis, commonly done using static sensor networks is a well studied problem in literature. 
	 Recently, due to flexibility in setting the spatial sampling density and low hardware cost, owing to larger spatial coverage, mobile sensors are used for this purpose. The main challenge in using mobile sensors is their location uncertainty. Obtaining location information of samples requires additional hardware and cost. So, we consider the case when the spatio-temporal field along the fixed length path is sampled using a simple mobile sensing device that records field values while traversing the path without any location information. We ask whether it is possible to learn the statistical distribution of the field, as a function of spatial location, using samples from the location-unaware mobile sensor under some simple assumptions on the field. We answer this question in affirmative and provide a series of analytical and experimental results to support our claim. 
	\end{abstract}

	
\section{Introduction}
Learning the statistical distribution of spatial (physical) fields is a
fundamental task in applications such as environment monitoring and pollution
management. Classical solution to this problem includes the usage of a fixed
array of sensors, deployed at prime locations over the region of
interest (done by agencies like EPA - \url{https://www.epa.gov/}). In the past decade, a thorough analysis and some testbeds have
proposed the usage of mobile sensors (see ~\cite{Unnikrishnan2013}). A key advantage of a mobile sensor
is that its infrastructure cost is lower and facilitates higher density of
samples. A major challenge in using mobile sensors is the uncertainty of sample
locations.  With a global positioning system (GPS) receiver, the location can be
estimated.  However, this information requires extra cost in terms of hardware
and battery. The knowledge of location can also have privacy
implications on the mobile sensor or its experiments (see ~\cite{Lin_Ye}). That is,
localization of sensors in MSNs can cost extra energy and
hardware (see ~\cite{Che_09,hu2004localization, Lin_Tao, Nguyen}). 

In this work,  a spatial field $X(s,t)$ is considered for analysis where  $s \in
{\cal P}$ is the location on a finite-length path ${\cal P}$ and  $t\in\rR$ is
time. It is assumed that the field is bounded in $[-b,b]$ for a finite $b > 0$
and is Lipschitz continuous; that is $|X(s,t) - X(s',t)| \leq \alpha |s -s'|$ for a
finite and fixed $\alpha > 0$.  Learning the statistical distribution of the
field $X(s,t)$ as a function of $s$ along the path $ {\cal P}$ using samples
from a location-unaware mobile sensor is of interest.  The sampling locations
are unknown and modeled by an unknown renewal process as done in a
location-unaware mobile sensing setup in ~\cite{kumar17}.  Let sampling locations
be $S_1, S_2, \ldots, S_M$ along the path ${\cal P}$, where $M$ is the random
number of samples obtained during a mobile-sensing experiment. 
%
%
The inter-sample intervals $\theta_1 := S_1$, $\theta_2 := S_2 - S_1, \ldots$
are independent and identically distributed in our model.  By unknown
renewal process, it is implied that the distribution of $\theta_1$ is not known.

The proposed distribution-learning of the spatial field is designed around $N$
experiments of the mobile sensing experiment along path $ {\cal P} $.  It is
assumed that the obtained samples of the spatial-field, as well as the unknown
randomly distributed sampling locations,  are statistically independent between
different experiments. The samples from the same experiment may be dependent or
independent, but this assumption is not used in our subsequent analysis.  Using
the location-unaware samples from $N$ independent experiments, the statistical
distribution of $X(s, t)$ for any point $s\in {\cal P}$ has to be learned.
The main results of this work are as follows:
\begin{enumerate}
	\item We design a cumulative distribution function (CDF) estimand for $X(s, t),
	\quad s \in {\cal P}$ with the maximum error between the estimand and the
	underlying CDF of $X(s, t),  s \in {\cal P}$ is of the order
	$\mathcal{O}\left(\frac{1}{\sqrt n}\right)$.  This result holds when the number of
	experiments $N \rightarrow \infty$.
	\item When the number of experimental trials $N$ is finite, we design a CDF estimand for $X(s, t)$, $s \in {\cal P}$ such that the maximum error between the estimand and the
	underlying CDF of $X(s, t)$, $s \in {\cal P}$ is of the order
	$\mathcal{O}\left(\frac{1}{\sqrt n} + \frac{1}{\sqrt N}\right)$ with probability at least $1-\delta$ for any $\delta > 0$. 
	\item Using custom experiments with a handheld sound-meter and acoustic noise
	measurements along a bounded length path, a data set was created to validate the
	distribution learning method.  The experimental results compare
	distribution-learning with a fixed and a mobile sound-level probe.
\end{enumerate}
\begin{figure*}[!t]
	\centering
	\subfloat[]{\includegraphics[width=2in]{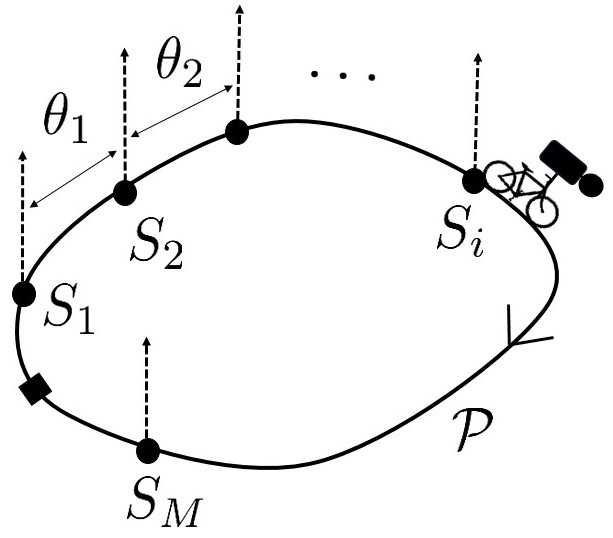}%
		\label{fig:intro_diag}}
	\hfil
	\subfloat[]{\includegraphics[width=2in]{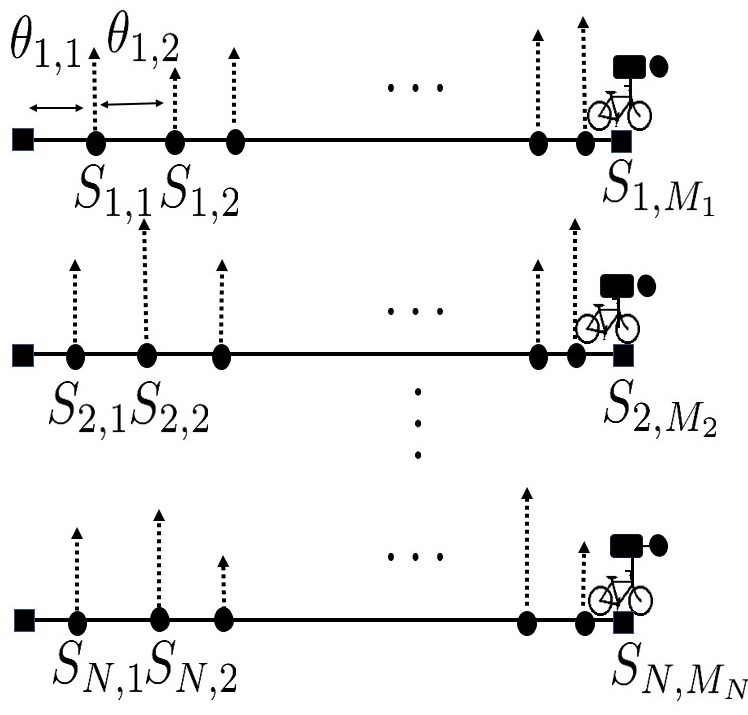}%
		\label{fig:sampling_model}}
	\caption{(a) Mobile
		sensor moving along the fixed length 1-D path ${\cal P }$. The
		field samples are obtained at unknown locations $S_1$, $S_2, \ldots, S_M$. (b) N
		trials of mobile sensing experiment are carried on N different days. Spatial field values are recorded at
		unknown locations $S_{i,1}, S_{i,2}, \ldots, S_{i,M_{i}}$ on trial $i$. The
		number of samples recorded during the $i$th trial is denoted by $M_i$.}
	\label{fig_sim}
\end{figure*}

\textit{Prior art}: Sampling and reconstruction of spatial fields is an extension of the classical
Shannon's sampling theorem addressed in \cite{A.J.Jerri1977,Marvasti2012}.
Sensing of spatial fields with static/fixed sensor networks is a broadly studied topic 
studied in \cite{Marco2003,Nowak2004,Kumar2011,Willet2004}. A systematic analysis 
of spatial field sampling using mobile sensors is reported in \cite{Unnikrishnan2013}. 
\\\\Location-unaware sensing and analysis of spatial fields in the case of static sensors
as well as mobile sensors is addressed in literature. Sampling of a bandlimited spatial
field at unknown locations is reported in \cite{Marziliano2000,Nordio2008,Elhami2018}.
Sampling of spatially
bandlimited fields at unknown but statistically distributed sampling locations
is addressed in~\cite{Kumar2015}, and~\cite{Mallick2016} studies the optimal
probability distribution of sensor locations in this case.  It is proved
in~\cite{Pacholska2017} that reconstruction of polynomial functions can be
achieved if unknown sampling locations can be described by an unknown rational
function.  Bandlimited field reconstruction using location-unaware mobile sensing
is addressed in \cite{Kumar2016, kumar17}. Spatio-temporal field sampling
using location-unaware mobile sensors, where a one dimensional spatial field evolving
with time according to a partial differential equation is studied in \cite{salgia2018}.

In addition to the sampling and reconstruction of spatial fields using a dedicated sensor 
network, as mentioned above, a study on crowd sensing approach to spatial field sensing 
is available in literature. Environmental monitoring with smart-phone applications is 
studied in \cite{Picaut2019} and vehicular sensor networks is studied in \cite{Wang2017},
\cite{Morselli2018}. 

When the spatial field is random, instead of reconstructing the field, it is of interest 
to learn the statistical distribution of the spatial field as a function of spatial location.
Consider the spatial field $X(s, t)$ varying with time and space. At a given location $s_0$, the
field value $X(s_0, t)$ is random. It is of interest to learn the distribution of this random
variable. The problem of learning the distribution of a random variable from its samples is 
a well studied problem in literature (see \cite{Kamath2015, orlitsky2015, SURF}).
However, in our problem, since the sensor is location-unaware, the field is not sampled 
exactly at the desired location. In such a case, the following sections provide a method
to learn the distribution of the spatial field under some simple assumptions.

\section{Spatial Field Properties}
Let $ X( s , t ) $ be the spatial field of interest, where $ s $ denotes a
spatial location and $ t $ denotes time.  Time $ t \in \rR $ and location $ s
\in \cal P $, where $\cal P $ denotes a bounded one-dimensional path. Path $
\cal P $ can be parameterized by a scalar and it can be normalized to $ { \cal P
} = [ 0, 1 ]$. The field is assumed to be bounded in $ [ -b, b ]  $  for
a finite $ b > 0 $ and Lipschitz continuous w.r.t. $ s $;
that is 
\begin{align*} 
	|X(s_1, t) - X(s_2, t)| \leq \alpha |s_1 - s_2|,
\end{align*} 
where $\alpha > 0$ is the Lipschitz constant. Lipschitz continuity indicates that nearby spatial locations
have similar field values. The field $ X( s , t ) $ may not be stationary.

\section{Location-Unaware Mobile Sensing Model}

The field $ X( s, t ) $ is sampled using a location-unaware mobile sensor. The
sensor moves along the path $ \cal P $ and samples the field at regular intervals in
time. As the sensor is location-unaware and moves at non-uniform speed, samples
are obtained at unknown locations $ S_1, S_2, \ldots, S_M $ along the path
${\cal P}$. The inter-sample
intervals are $\theta_1 := S_1$, $\theta_2  := S_2 - S_1, \ldots$, $\theta_M :=
S_M  - S_{M-1}$. Depending on the time taken to traverse the path $\cal P$, a
random number of samples is recorded during a mobile sensing trial. The random
number of samples denoted by $M$ is given by a stopping rule (see ~\cite{durrett}):
\begin{align*}
	\theta_{1} + \theta_{2} + \ldots + \theta_{M} \leq 1 \mbox{ and  } \theta_{1} +
	\theta_{2} + \ldots + \theta_{M+1} > 1.
\end{align*}

The sampling locations are modeled as arrivals of an unknown renewal process.
Therefore, $ \theta_1, \theta_2, \ldots $ are independent and identically
distributed. It is assumed that 
\begin{align*}
	0<\theta\leq\frac{\lambda}{n} \quad \mbox{ and } \quad \eE(\theta)=\frac{1}{n},
\end{align*}
where $ \lambda > 1 $ is finite and determined by the maximum sensor speed. The
average spatial sampling rate of the sensor is $ n $.  It is shown in~\cite{kumar17} that the conditional average of $\theta$ conditioned on $M=m$
is approximately $\frac{1}{m}$ and 
\begin{align*}
	\eE[M]\leq n + \lambda -1. 
\end{align*} 
%
%

%
%
A model of mobile sensing experiment is explained next. 
The field samples are recorded by performing $ N $ independent trials of mobile
sensing along the path $ \cal P $. Let $ \vec{ X }_i  $ be the set of field samples
collected and let  $ M_i $ be the number of samples
recorded  during trial $i$. After $ N $ trials we obtain 
\begin{align*}
	\vec{ X }_1 &:= [X(S_{1,1}, t_{1,1}),  X(S_{1, 2}, t_{1, 2}), \ldots, X(S_{1,
		M_1}, t_{1, M_1})]^T; \\
	\ldots, \nonumber \\
	\vec{ X }_N &:= \big[X(S_{N, 1}, t_{N, 1}),  X(S_{N, 2}, t_{N, 2}), \\ & \hspace{4.2cm} \ldots, X(S_{N,
		M_N}, t_{N, M_N})\big]^T.
\end{align*}
The location and time of the $ j $th field value recorded during trial $ i $ are
given by $ S_{ i, j } $ and  $ t_{ i, j } $ respectively. These sampling
locations for $ N $ different trials are generated by the same unknown renewal
process. The field values $ \vec{X}_1, \vec{X}_2 \ldots $ are statistically
independent; however, the field values in $\vec{X}_i$ for a given trial may be
dependent.
\section{Distribution Learning of Spatial Field}
It is of interest to learn the distribution of the field $ X( s, t) $ at a given
location $ s \in [ 0, 1] $. The field value $ X( s , t_0 ) $ at location $ s $,
at time $ t_0 \in \rR $ is estimated by $ \hat{ X }( s ) $ which is given by, 
\begin{align*}
	\hat{X}(s)= X\left(s_{ \lfloor ( M - 1 )s \rfloor+1}, t_{ 0 }\right),
\end{align*}
where $M$ is the random number of samples recorded along path $ \cal P $ and  $s_{
	\lfloor ( M - 1 )s \rfloor+1}, t_0 $ are the location and time at which $ {
	(\lfloor ( M - 1 )s \rfloor+1} )^{th} $ sample is recorded respectively. The
field samples $ \vec{X}_1, \vec{X}_2 \ldots  \vec{X}_N $ are available after $N$
trials of the mobile sensing experiment. The field value at location $ s \in {[ 0, 1 ] } $
for the $i^{th}$ trial is estimated by 
\begin{align}\label{x_i_hat}
	\hat{X}_i(s)= X\left(s_{ i, {\lfloor ( M_i - 1 )s \rfloor+1} }, t_{i, \lfloor (M_i -
		1)s \rfloor + 1} \right), 
\end{align}
while $ X_{i}(s) = X\left(s, t_{i, \lfloor (M_i - 1)s \rfloor + 1}\right) $. For
simplicity, the dependence on $ t $ has been dropped on the left-hand side as
the distribution learning accuracy is governed by the error in estimating sample
locations.
\begin{prop}\label{prop_1}
	\textit{(\cite{NIPS2019}, Theorem 1, Equation (11))}

	Consider the path ${\cal P} = [0, 1]$. Let $S_1, S_2, \ldots, S_M$ be the unknown
	sampling locations, and $\theta_1, \theta_2, \ldots, \theta_M$ be the
	inter-sample intervals. The sampling locations are points from an 
	unknown renewal process which satisfies $\eE[
	\theta_1 ] = \frac{1}{n} $ and $ 0 < \theta \leq \frac{\lambda}{n} $. Then the
	mean squared error between sampling location $ S_{\lfloor (M-1)s\rfloor + 1} $
	and location $ s $ is
	\begin{equation*}
		\eE\left[ \left| S_{\lfloor (M-1)s \rfloor + 1}-s \right|^2 \right ] \leq ( ( n
		+ \lambda - 1 ) s (1-s) + C )\frac{ \lambda^2 }{ n^2 },
	\end{equation*}
	where $ C $ is a constant. 
\end{prop}
Proposition \ref{prop_1} shows that the mean squared error between location $s$ on path ${\cal P}$ and the estimated
sampling location decreases with increasing spatial sampling rate $n$.  The
field is assumed to be Lipschitz continuous, so
\begin{equation}\label{Lipschitz_property}
	\left|X(s)-\hat{X}(s)\right|\leq \alpha \left|S_{\lfloor (M-1)s\rfloor
		+1}-s\right|,
\end{equation}
where $ \alpha $ is the Lipschitz constant. This shows that for a given value of
$ \alpha $, the error between the field value at location $ s $, $ X(s) $ and
the estimated field value at location $ s $, $ \hat{ X }(s) $ at a given time,
depends only on the error in sampling location. Using the Lipschitz property of the field
and Proposition \ref{prop_1}, a bound on the probability of error in estimating $X(s)$ by $\hat{X}(s)$
is given in \cite{NIPS2019}.
\begin{prop}\label{prop_2}
	\textit{(\cite{NIPS2019}, Theorem 1, Equation (14) and Theorem 2, Equation (19))}	
	Let $\hat{ X }(s)$ be the estimate of the field value at location $s$, estimated using field value sampled at unknown sampling location $S_{ \lfloor ( M - 1 )s \rfloor+1}$. Let  $ M $ be the random number of samples generated by the location-unaware mobile sensor on a finite length path ${\cal P}$. The sampling locations $ S_1, S_2, \ldots, S_M $ are generated by an unknown renewal process, where the inter-sample intervals $ \theta_1, \theta_2, \ldots \theta_M $ are such that $ \eE[\theta_1 ] = \frac{1}{n} $ and $ 0 < \theta \leq \frac{\lambda}{n} $. Then for every $s \in [0,1]$ and for any $\varepsilon>0$,
	\begin{align} \label{bound_on_prob_prop2}
		\pP \left( \left| X(s) - \hat{ X }(s)  ) \right| > \varepsilon \right ) \leq \frac{ \alpha^2 }{ \varepsilon^2 } ( ( n + \lambda - 1)s(1-s) + C )\frac{
			\lambda^2 }{ n^2 }  
	\end{align}
	and 
	\begin{equation}\label{bound_on_sup_prob_prop2}
		\pP\left(\sup_{s\in[0,1]}\left|\hat{X}(s)-X(s)\right|>\varepsilon\right)
		\leq\frac{32}{\beta}\frac{\alpha^2}{\varepsilon^2}(n+\lambda-1)\frac{\lambda^2}{n^2},
	\end{equation}
	where $\beta$ goes to 1 as $n \to \infty$.
\end{prop}
From Proposition \ref{prop_2}, ~\eqref{bound_on_prob_prop2} we can see that for a given accuracy level $\varepsilon>0$, probability of the error(absolute difference) between the field value $ X(s) $ and its estimate $ \hat{ X }(s) $ at location $ s $, at a given time, being greater than $\varepsilon$ has an upper bound of order $\mathcal{O}(\frac{1}{n})$. For a  given $\varepsilon$ this upper bound decreases with increasing spatial sampling rate $n$. In other words, as the sampling rate $n$ tends to infinity, $\hat{X}(s)$ converges to ${X}(s)$. The upper bound in $\eqref{bound_on_sup_prob_prop2}$ implies that as the sampling rate $n$ tends to infinity, $\hat{X}(s)$ converges \textit{uniformly over $s\in[0,1]$} to ${X}(s)$.
Having obtained the estimated field values $\hat{X}_1(s), \hat{X}_2(s), \ldots, \hat{X}_N(s)$ at location $s$ after $N$ trials, we estimate the cumulative distribution function(CDF) of field value at location $s$, denoted by $F_{\hat{X}(s)}(x)=\pP(\hat{X}(s)\leq x)$. Using the classical
Glivenko-Cantelli theorem, the estimated CDF at location $s$, $F_{\hat{X}(s)}(x)$ is obtained as 

\begin{align}\label{def_F_estimate}
	F_{\hat{X}(s)}(x) = \lim_{N \rightarrow \infty} \frac{1}{N} \sum_{i = 1}^N
	\mathbb{I} \left(\hat{X}_i(s) \leq x \right),
\end{align}
where $\mathbb{I}(.)$ denotes the indicator function. Let CDF of the field value at location $s$
be denoted by
\begin{align*}
	F_{X(s)}(x)=\pP(X(s)\leq x).
\end{align*}
Now, we investigate the difference between the CDF of field value $F_{X(s)}(x)$ at location $s$ and its estimate $F_{\hat{X}(s)}(x)$.
\begin{theorem}\label{theorem_1}
	Let $ M $ be the random number of samples generated when a location-unaware
	mobile sensor samples the field along the finite length path ${\cal P}$. Let $ S_1, S_2, \ldots, S_M $ be the
	sampling locations and $ \theta_1, \theta_2, \ldots \theta_M $ be the
	inter-sample intervals generated by the unknown renewal process such that $ \eE[
	\theta_1 ] = \frac{1}{n} $ and $ 0 < \theta \leq \frac{\lambda}{n} $. Then for
	every $x\in \rR, s \in [0,1]$ and for any $\varepsilon>0$,
	\begin{align}
		 \lvert F_{X(s)}(x) - F_{\hat{X}(s)}(x)\rvert 
		 \leq \frac{ \alpha^2 }{ \varepsilon^2 } ( ( n + \lambda - 1 )s( 1 - s ) + C
		)\frac{ \lambda^2 }{ n^2 } 
		 + 2\alpha\max_{x}( f_{X(s)}( x
		))\frac{\lambda}{n}\sqrt{((n+\lambda-1)s(1-s)+C)}. \label{upper_bound_theorem1} 
	\end{align}
\end{theorem}

\begin{proof} We can split $F_{\hat{X}(s)}(x)$ as
	\begin{align*}
		& F_{ \hat{X}(s) }( x ) \nonumber \\
		& = \pP( \hat{X}(s) \leq x)\\
		& = \pP( \hat{X}(s) \leq x, X(s) \leq x) + \pP( \hat{X}(s) \leq x, X(s) > x)\\
		& \leq \pP( X(s) \leq x ) + \pP( \hat{X}(s) \leq x, X(s) > x)\\
		& \leq F_{X(s)}(x) + \pP( \hat{X}(s) \leq x, X(s) > x + \varepsilon)  
		 + \pP( \hat{X}(s) \leq x, x < X(s) \leq x + \varepsilon).
	\end{align*}
	Therefore, 
	\begin{align} \label{up2_thm1}
		&  \lvert F_{ \hat{X}(s) }( x ) - F_{X(s)}(x) \rvert \leq    \underbrace{\pP(  | \hat{X}(s)-X(s) |  >
			\varepsilon)}_{ft} 
		 + \underbrace{\pP( \hat{X}(s) \leq x, x < X(s) \leq x +
			\varepsilon)}_{st}.
	\end{align}
	An upper bound on ($ft$) is given in Proposition \ref{prop_2}, \eqref{bound_on_prob_prop2}. We propose the following lemma that gives an upper bound on ($st$). 
	\begin{lemma}\label{lemma_1}
		Let $\hat{ X }(s)$ be the estimate of the field value at location $s$, estimated using field value sampled at unknown sampling location $S_{ \lfloor ( M - 1 )s \rfloor+1}$. Let  $ M $ is the random number of samples generated by the location-unaware mobile sensor on a finite length path ${\cal P}$. The sampling locations $ S_1, S_2, \ldots, S_M $ are generated by an unknown renewal process, where the inter-sample intervals $ \theta_1, \theta_2, \ldots \theta_M $ are such that $ \eE[\theta_1 ] = \frac{1}{n} $ and $ 0 < \theta \leq \frac{\lambda}{n} $. Then for
		every $x\in \rR, s \in [0,1]$ and for any $\varepsilon>0$,
		\begin{align*} 
			& \pP( \hat{X}(s) \leq x, x < X(s) \leq x + \varepsilon) 
			\leq 2\alpha\max_{x}( f_{X(s)}( x
			))\frac{\lambda}{n}\sqrt{((n+\lambda-1)s(1-s)+C)}, 
		\end{align*}
		where $ f_{X(s)} $ is the probability density function of the field value at
		location $ s $.
	\end{lemma}
	\begin{proof}
		A detailed proof of Lemma \ref{lemma_1} is given in Appendix \ref{proof_lemma_1}. Owing to the Lipschitz nature of the field, from \eqref{Lipschitz_property} it holds that if $\hat{X}(s) \leq x$, then  $X(s) \leq x + \alpha\left|S_{\lfloor (M-1)s\rfloor
			+1}-s\right|$. So,
		\begin{align*} 
			\pP( \hat{X}(s) \leq x, x < X(s) \leq x + \varepsilon) \leq \pP( x < X(s) \leq x + Z), \\
		\end{align*} 
		where $Z = \min\left(\varepsilon, \alpha\left|S_{\lfloor (M-1)s\rfloor +1}-s\right|\right)$. From Proposition \ref{prop_1}, the mean squared error between $s$ and its estimate $S_{\lfloor (M-1)s\rfloor +1}$ is of order $\mathcal{O}(\frac{1}{n})$. Thus it is shown in Appendix \ref{proof_lemma_1}
		that $\eE[Z]$ has an upper bound of order $\mathcal{O}(\frac{1}{\sqrt{n}})$ and $\pP( x < X(s) \leq x + Z)$ is bounded from above by $2\alpha\max_{x}( f_{X(s)}( x
		))\frac{\lambda}{n}\sqrt{((n+\lambda-1)s(1-s)+C)}$, which is independent of $\varepsilon$. 
	\end{proof}
	The Lemma follows by substituting the upper bound on ($ft$) and ($st$) in \eqref{up2_thm1}. 
\end{proof}
The upper bound given in the RHS of \eqref{upper_bound_theorem1} has two terms. The first term increases with decreasing value of $\varepsilon$ and the second term is of order $\mathcal{O}\left(\frac{1}{\sqrt n}\right)$. Thus, Theorem \ref{theorem_1} implies that the upper bound on the pointwise absolute difference between the CDF of field value, $F_{X(s)}(x)$ and its estimate, $F_{ \hat{X}(s) }( x )$,  at location $s$, is at least of order $\mathcal{O}\left(\frac{1}{\sqrt n}\right)$. The value of $\varepsilon$ can be chosen such that the upper bound in \eqref{upper_bound_theorem1} is of order $\mathcal{O}\left(\frac{1}{\sqrt n}\right)$. For example if $\varepsilon = n^{-\frac{1}{4}}$,
\begin{align*}  
	 \lvert F_{ \hat{X}(s) }( x ) - F_{X(s)}(x) \rvert 
	 & \leq 
	\alpha^2\sqrt{n}  ( ( n + \lambda - 1 )s( 1 - s ) + C )\frac{
		\lambda^2 }{ n^2 } \nonumber \\
	 & \hspace{5.3cm} + 2\alpha\max_{x}( f_{X(s)}( x
	))\frac{\lambda}{n}\sqrt{((n+\lambda-1)s(1-s)+C)}. 
\end{align*}
As the spatial sampling rate $n$ tends to infinity, the upper bound on the error in estimation of CDF in \eqref{upper_bound_theorem1} tends to zero for some $\varepsilon >0$. Thus, $F_{\hat{X}(s)}(x)$ converges to $F_{X(s)}(x)$ as $n$ tends to infinity.  
The upper bound given in the RHS of \eqref{upper_bound_theorem1} depends on $s$ and has a maximum at $s=1/2$. 
By taking supremum over $ s $ we get,
\begin{align*}  
	 \sup_{s\in[0,1]}\lvert F_{\hat{X}(s)}(x)-F_{X(s)}(x)\rvert 
	 \leq 
	\frac{ \alpha^2 }{ \varepsilon^2 } \left ( \frac{( n + \lambda - 1 )}{4} + C
	\right )\frac{ \lambda^2 }{ n^2 }
	 + 2\alpha\max_{x,s}( f_{X(s)}( x ))\frac{\lambda}{n}\sqrt{ \left (\frac{( n +
			\lambda - 1 )}{4}+C \right )}.
\end{align*}
This implies that as the sampling rate $n$ tends to infinity, for some $\varepsilon >0$, $F_{\hat{X}(s)}(x)$ converges uniformly over $s\in[0,1]$ to $F_{X(s)}(x)$.

Theorem \ref{theorem_1} gives an upper bound on the error between $F_{X(s)}(x)$ and its estimate $F_{\hat{X}(s)}(x)$, defined in \eqref{def_F_estimate} for number of trials $N \to \infty$. However, the number of trials $N$ is finite in practice. The field values $ \vec{X}_1, \vec{X}_2, \ldots, \vec{X}_N $  are obtained after
$N$ trials of the mobile sensing experiment. The field value at a given location
$s$ on path ${ \cal P }$, for trial $i$ is estimated by $\hat{X_i}(s)$ given in
\eqref{x_i_hat}. Using $N$ estimates of the field value, the empirical CDF at
location $s$ is 
\begin{align}\label{empirical_finite}
	F_{\hat{X}(s),N}(x) = \frac{1}{N} \sum_{i = 1}^N \mathbb{I} \left(\hat{X}_i(s)
	\leq x \right).
\end{align}
\begin{theorem}\label{theorem_2}
	Let $ M $ be the random number of samples generated when a location-unaware
	mobile sensor samples the field along the finite length path ${\cal P}$. Let $ S_1, S_2, \ldots, S_M $ be the
	sampling locations and $ \theta_1, \theta_2, \ldots \theta_M $ be the
	inter-sample intervals generated by the unknown renewal process such that $ \eE[
	\theta_1 ] = \frac{1}{n} $ and $ 0 < \theta \leq \frac{\lambda}{n} $. Then for
	every $x\in \rR, s \in [0,1]$ and for any $\delta>0$, 
	\begin{align}
		 \sup_{x \in \rR } \lvert F_{\hat{X}(s),N}(x) - F_{X(s)}(x) \rvert 
		 \leq \frac{A^2}{\varepsilon^2} + 2\max_{x}( f_{X(s)}(x))A+ \sqrt{\frac{\log(\frac{2}{\delta})}{2N}} \label{upper_bound_thorem2}
	\end{align}
	with probability at least $1-\delta$, where $ A = \alpha\sqrt{((n+\lambda-1)s(1-s)+C)}\frac{\lambda}{n} $.
	
\end{theorem}
\begin{proof}
	The error in estimation of $F_{X(s)}(x)$ using finite number of samples is 
	\begin{align}
		\pP \left( \sup_{x \in \rR } \lvert F_{\hat{X}(s),N}(x) - F_{X(s)}(x) \rvert >
		\xi \right)  
		 \leq \pP \Big( \sup_{x \in \rR } \lvert F_{\hat{X}(s),N}(x) - F_{\hat{X}(s)}(x) \rvert  
		 + \sup_{x \in \rR } \lvert F_{\hat{X}(s)}(x) - F_{X(s)}(x) \rvert >
		\xi \Big). \label{error_in_dist}
	\end{align} 
	The upper bound on error in CDF estimation given in Theorem \ref{theorem_1}, \eqref{upper_bound_theorem1} is $
	\frac{A^2}{\varepsilon^2} + 2\max_{x}( f_{X(s)}( x ))A $, where $ A =
	\alpha\sqrt{((n+\lambda-1)s(1-s)+C)}\frac{\lambda}{n} $. Substituting the upper bound in
	\eqref{error_in_dist} gives,
	\begin{align*}
		\pP \left( \sup_{x \in \rR } \lvert F_{\hat{X}(s),N}(x) - F_{X(s)}(x) \rvert >
		\xi \right) 
		 \leq \pP \Big( \sup_{x \in \rR } \lvert F_{\hat{X}(s),N}(x) - F_{\hat{X}(s)}(x) \rvert 
		 > \xi -\frac{A^2}{\varepsilon^2} - 2\max_{x}( f_{X(s)}( x ))A \Big).
	\end{align*}
	By applying DKW inequality (from \cite{massart}) to RHS of above equation we get
	\begin{align*}
		\pP \left( \sup_{x \in \rR } \lvert F_{\hat{X}(s),N}(x) - F_{X(s)}(x) \rvert >
		\xi \right) 
		 \leq 2e^{ - 2 N {\left( \xi -\frac{A^2}{\varepsilon^2} - 2\max_{x}( f_{X(s)}(
				x ))A \right )}^2 }.
	\end{align*}
	If $\delta = 2e^{ - 2 N {\left( \xi -\frac{A^2}{\varepsilon^2} - 2\max_{x}( f_{X(s)}(
			x ))A \right )}^2 } $, then $\xi = \frac{A^2}{\varepsilon^2} + 2\max_{x}( f_{X(s)}(x))A + \sqrt{\frac{\log(\frac{2}{\delta})}{2N}}$. Thus,
	\begin{align*}
		& \pP \left( \sup_{x \in \rR } \lvert F_{\hat{X}(s),N}(x) - F_{X(s)}(x) \rvert > \xi
		\right) \leq \delta.
	\end{align*}
	Thus, the theorem follows from the above inequality. 
\end{proof}	
The upper bound on the RHS of \eqref{upper_bound_thorem2}, in Theorem \ref{theorem_2}, comprises of three terms. The first two terms constitute the upper bound on the pointwise absolute error between $F_{X(s)}(x)$ and its estimate $F_{\hat{X}(s)}(x)$ given in Theorem \ref{theorem_1}. The value of $\varepsilon$ can be chosen such that it is of order $\mathcal{O}(\frac{1}{\sqrt{n}})$. The third term is of order $\mathcal{O}(\frac{1}{\sqrt{N}})$. Thus, Theorem \ref{theorem_2} implies that, the supremum of the pointwise absolute error between $F_{X(s)}(x)$ and its estimate $F_{\hat{X}(s),N}(x)$, taken over $x$, is less than an upper bound of order $\mathcal{O}\left(\frac{1}{\sqrt{n}} + \frac{1}{\sqrt{N}}\right)$ with probability atleast $1-\delta$. Thus, the error in estimating $F_{X(s)}(x)$ by $F_{\hat{X}(s),N}(x)$ decreases with the increasing spatial sampling rate $n$ and the number of trials $N$. Increasing the spatial sampling rate leads to better estimation of field values, whereas increasing the number of trials reduces the error in learning the CDFs using their empirical estimates. 

In Section \ref{simulation} we verify our claim by running our distribution learning method on a simulated spatio-temporal field. Using a custom generated dataset we verify the performance of our distribution learning method in section \ref{experiment}.   

\section{Simulation of Distribution Learning Method with Samples from Location-Unaware Mobile Sensor\label{simulation}} 
In this section, we apply our distribution learning method to learn the distribution of a simulated spatio-temporal field $X(s, t)$. We sample this field at locations generated by a renewal process and estimate the CDF of field value at some point $s\in [0,1]$ using these samples.  
A spatio-temporal field $ X(s, t)$ is generated, where 
\begin{equation*}
	X(s,t) =  a_0  + \sum_{k=1}^{5}  a_k(t) cos(2\pi kfs).
\end{equation*}
The constants $ a_0 = 500 $, $f=5$ and $ a_k $  is a uniform random variable whose
parameter is a function of time $t$. The field $X(s,t)$ is a summation of cosine functions 
varying with spatial locations having time varying amplitudes. It is of interest to learn 
the CDF of field value at
location $s$, $F_{X(s)}$. However, $N$ trials of the location-unaware mobile
sensing experiment are simulated and we can only learn the empirical CDF
$F_{X(s), N}$ using $N$ samples of $X(s,t)$ at location $s$. Let $m_i$ be the
random number of samples generated during trial $i$, where $i =1,2, \ldots, N$.
Let $s_{i, 1}, s_{i, 2}, \ldots, s_{i, m_i} $ be the randomly generated sampling
locations and $t_{i, 1}, t_{i, 2}, \ldots, t_{i, m_i} $ be the sampling times
for trial $i$. Sampling location for trial $i$ and sample $j$ is $s_{i, j} =
\sum_{l=1}^{j}\theta_{i, l} $, where $\theta_{i, l}$ is the intersample interval
generated by a Triangular distribution with mean $\frac{1}{n}$, lower limit $0$
and upper limit $\frac{2}{n}$.  Let $k = \lfloor (M-1)s \rfloor +1$. For the
$i^{th}$ trial, the sample at sampling location $s_{i, k}$, $ \hat{X}_i(s) = X(
s_{ i, k}, t_{ i , k} )$ is the estimate of the field value $ X_i(s) = X( s, t_{
	i, k} )$ at location $s$. We compare the empirical CDFs $F_{X(s),N}(x) =
\frac{1}{N} \sum_{i = 1}^N \mathbb{I}\left(X_i(s) \leq x \right)$ and
$F_{\hat{X}(s),N}(x)$ given in \eqref{empirical_finite}.

Location-unaware mobile sensing experiment is simulated for $N=50$ and $N=500$
and for spatial sampling rates $n=10, 100, 1000$ and $10000$ samples/m. The
experiment is repeated 200 times for each value of $N$ and $n$.
Average absolute pointwise difference in CDF,
avg$_x{\left|F_{X(s),N}(x) - F_{\hat{X}(s),N}(x) \right|}$ and
maximum absolute pointwise difference in CDF, $
\max_{x}|F_{X(s), N }( x ) - F_{\hat{X}(s), N}(x) |$ are averaged over 200
iterations for each value of $N$ and $n$.
\begin{figure*}[!t]
	\centering
	\subfloat[]{\includegraphics[scale=0.6]{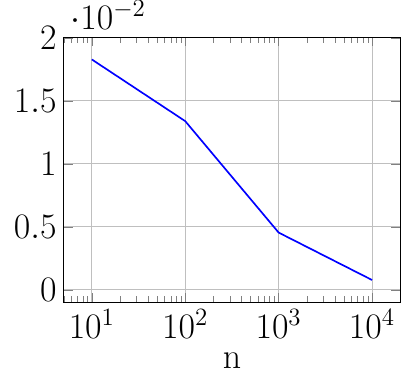}%
		\label{one}}
	\qquad \quad
	\subfloat[]{\includegraphics[scale=0.6]{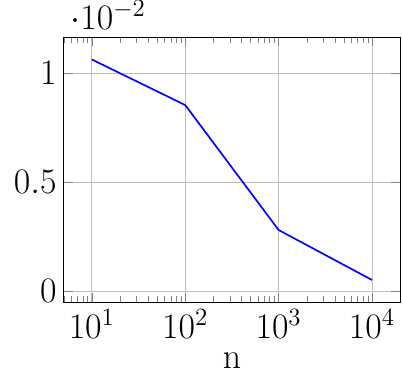}%
		\label{two}	}
	\caption{Average absolute pointwise difference in CDF,
		avg$_x{\left|F_{X(s),N}(x) - F_{\hat{X}(s),N}(x) \right|}$ plotted for different
		values of $n$ and  (a) $N=50$ and (b) $N=500$.}
	\label{avg_diff_cdf}
\end{figure*}
\begin{figure*}[!t]
	\centering
	\subfloat[]{\includegraphics[scale=0.6]{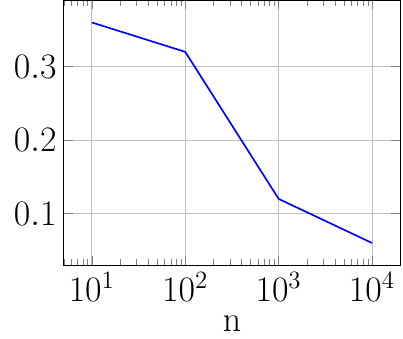}%
	}
	\qquad \quad
	\subfloat[]{\includegraphics[scale=0.6]{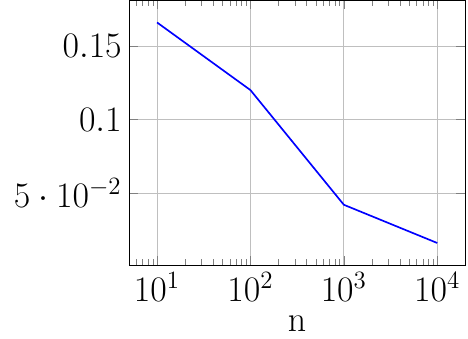}%
	}
	\caption{Maximum absolute pointwise difference in CDF, $
		\max_{x}|F_{X(s), N }( x ) - F_{\hat{X}(s), N}(x) |$ plotted for different
		values of $n$ and  (a) $N=50$ and (b) $N=500$.}
	\label{max_diff_cdf}
\end{figure*}

The results are plotted in Figure $\ref{avg_diff_cdf}$ and Figure $\ref{max_diff_cdf}$ where we observe that as
the spatial sampling rate increases the error in CDF estimation of field samples as a
function of spatial location decreases. We also observe that the error in CDF
estimation decreases with the number of trials $N$. 

\section{Acoustic Noise Distribution Learning Experiment with a Location-Unaware Mobile Soundmeter\label{experiment}}
In this section, an experiment is designed to validate the distribution learning method.
The mobile sensing experiment is performed along a fixed length path ${\cal
	P}_1$ shown in the map in Figure \ref{map}. The field value measured is the
acoustic-noise level along path ${\cal P}_1$. A sound level meter measuring
acoustic-noise level every second is carried along the path. Path ${\cal P}_1$
is a closed path along which there is a substantial variation in the acoustic-noise levels. It is not necessary for the path to be closed but it has to be of fixed length. Sound level meter by BAFX products (Model no: BAFX3608) is used in the
experiment. Its specifications are given in Table \ref{tab1}. It has memory to
store data but does not have any GPS or localization tool.  
\begin{table}[b] 
	\centering
	\caption{ \label{tab1} Specifications of Sound Meter.}
	\begin{tabular}{llll}
		\toprule \\
		Range: 30-130dB & 
		Sampling Rate: 1 per sec &\\
		Memory: 4700 readings &
		Accuracy: $\pm$ 1.5 dB   \\       
		\bottomrule
	\end{tabular}
\end{table}
\begin{figure}[h]
	\begin{center}
		\includegraphics[scale=0.15]{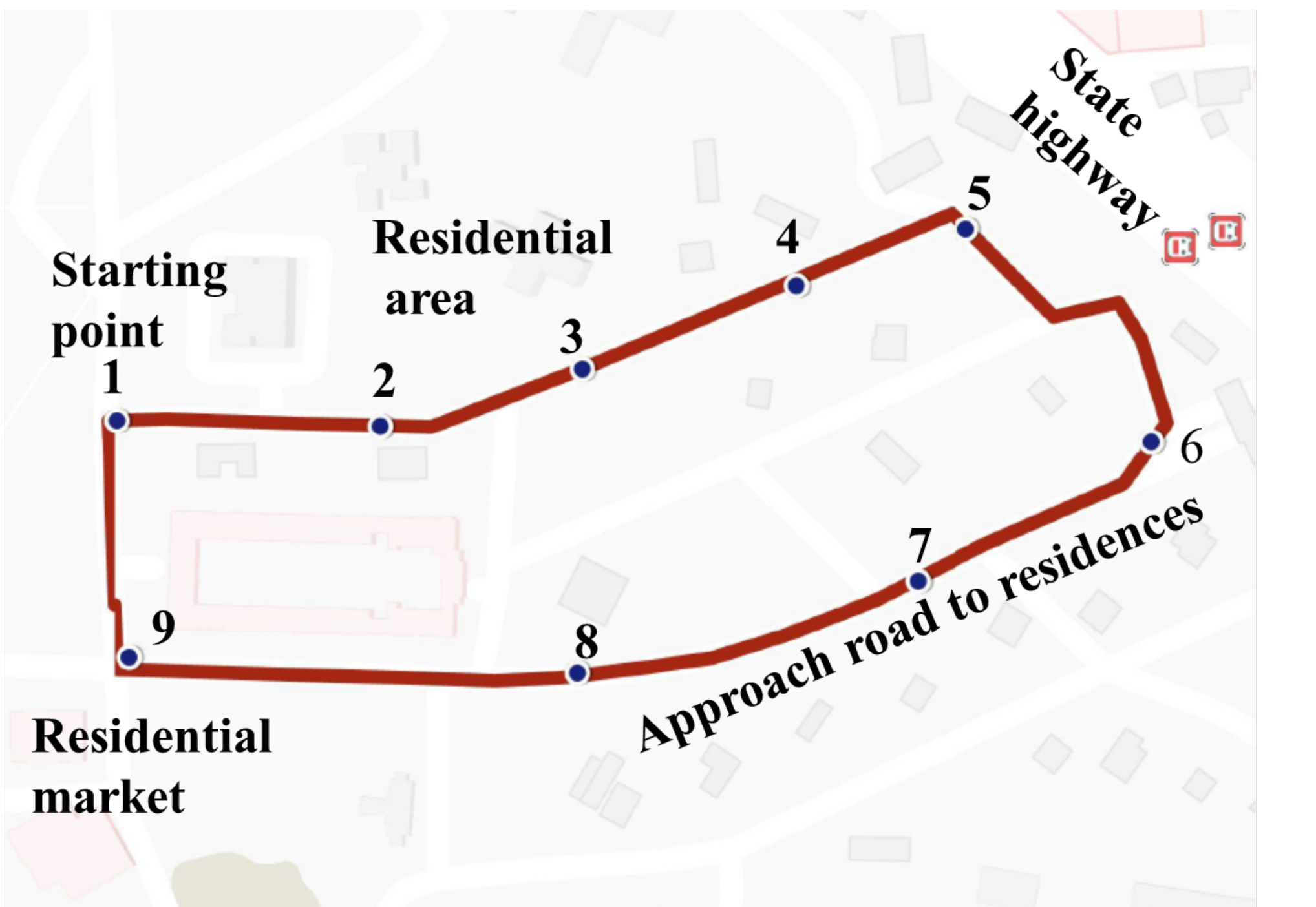}
		\caption{\label{map} Path along which acoustic noise levels are recorded.  The
			locations marked in the map with numbers are used for measurement using a fixed
			sensor.}
	\end{center}
\end{figure}
The sound level meter measures one sample per second. The distance between
locations of two different samples on ${\cal P}_1$ depends on the speed at which
the sound level meter traverses the path. Thus by varying the sensor speed,
the average spatial sampling rate $n$ of the sound level meter is varied. The sound
level meter is moved along path ${\cal P}_1$ at two different average speeds to
get two different data sets.
\begin{figure}[h]
	\begin{center}
		\centering
		\includegraphics[scale=0.25]{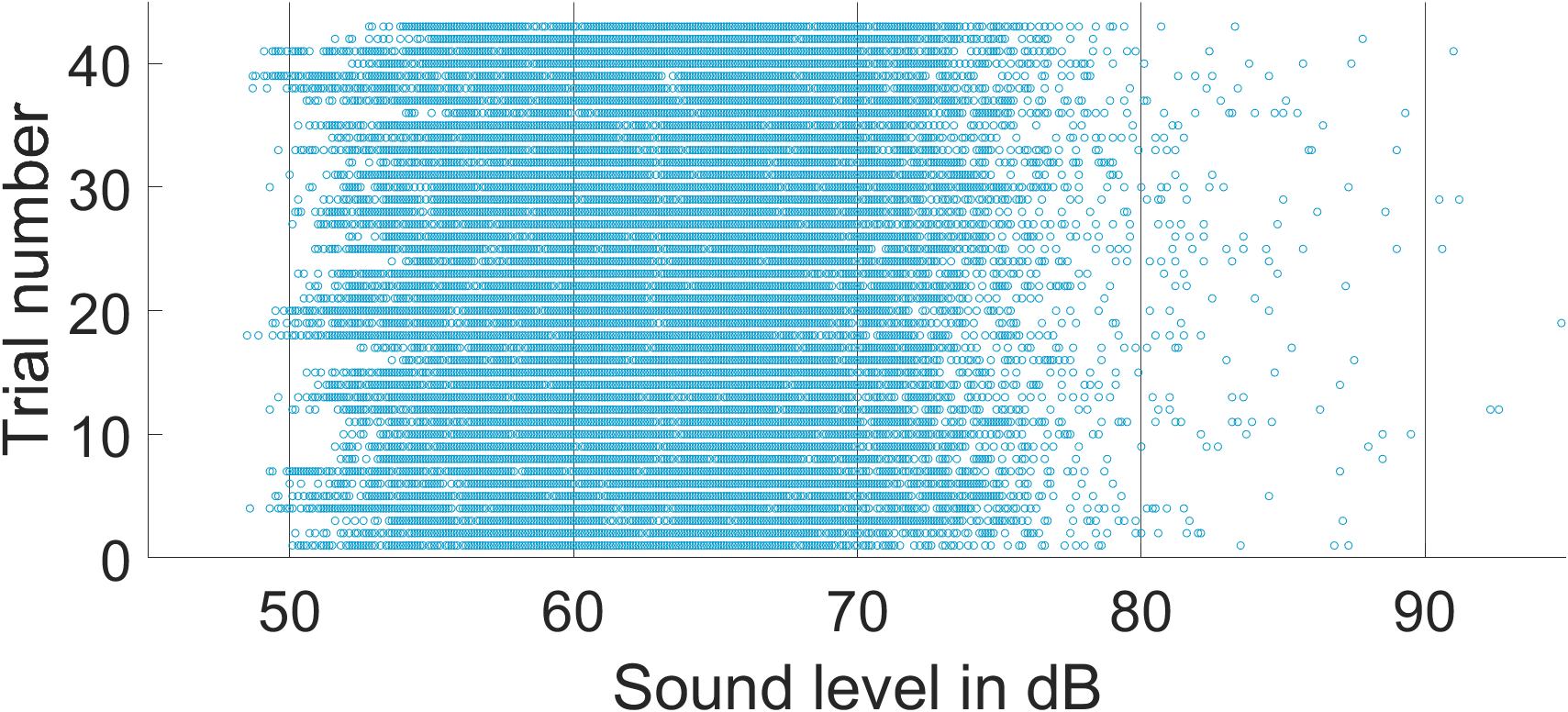}
		\caption{\label{scatterplot} Scatter plot of acoustic noise levels along the path ${\cal P}_1$ in
			Figure \ref{map} for each trial in \textit{Data set 1}.}
	\end{center}
\end{figure}
\begin{figure}[h]
	\begin{center}
		\centering
		\includegraphics[scale=0.23]{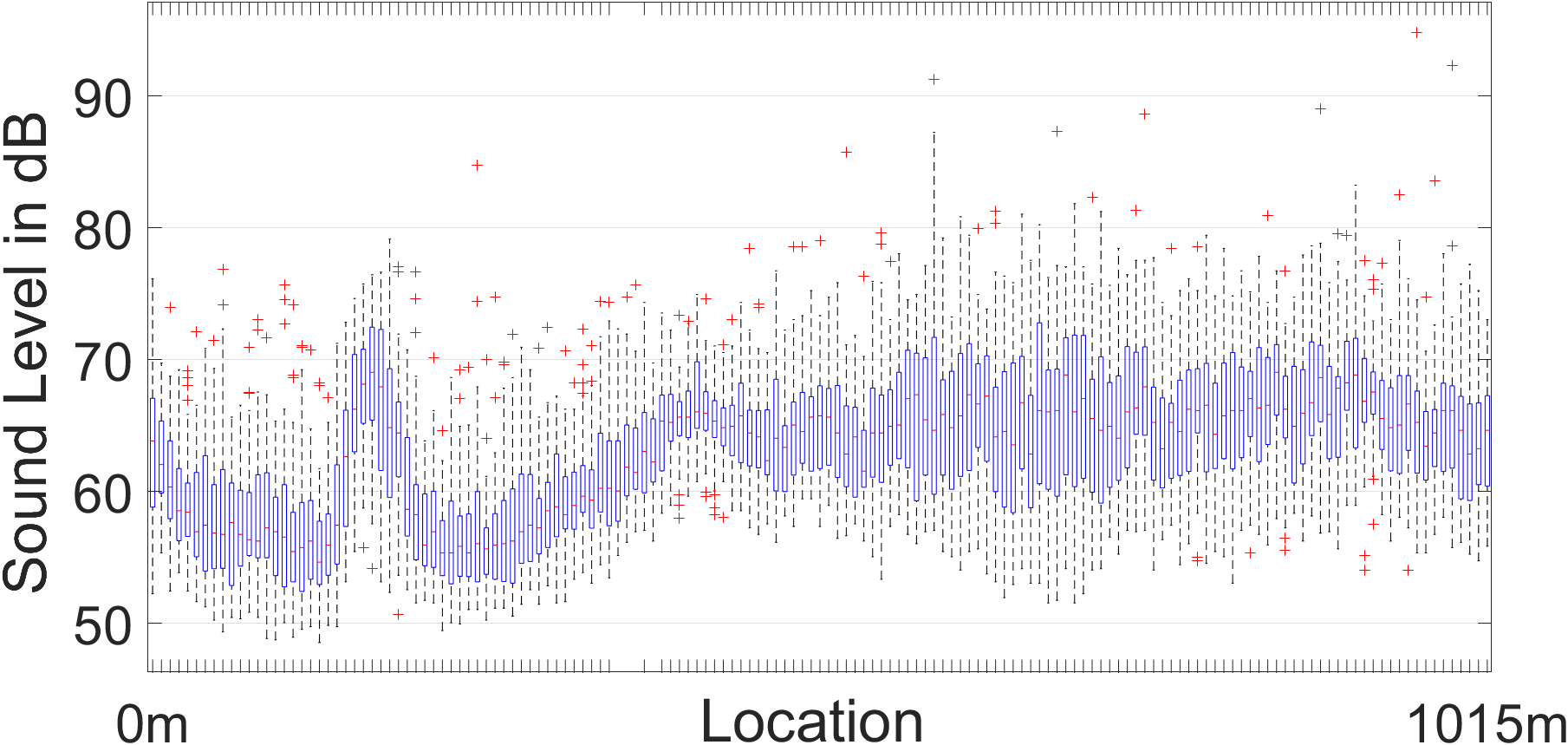}
		\caption{\label{boxplot} Box plot plotted using samples from the mobile sound
			level meter in \textit{Data set 1} along the path ${\cal P}_1$ in Figure \ref{map} of length 1015m.}
	\end{center}
\end{figure}
\textit{Data sets:}  \textit{Data set} 1 was created by walking along path ${\cal
	P}_1$ with the sound level meter. The trials of this experiment were performed 43
times along the same path on different days. Let $n_1$  be the average spatial
sampling rate for \textit{Data set} 1. \textit{Data set} 2 was created by cycling
along path ${\cal P}_1$ with the sound level meter. In this case as well 43 trials
were performed along the same path on different days. Let $n_2$  be the average
spatial sampling rate for \textit{Data set} 2. Faster the sensor speed, lower will
be the spatial sampling rate of the mobile sensor, sampling at the uniform rate
in time. Therefore,
\begin{equation*}
	n_2<n_1.
\end{equation*}

The number of trials performed N=43. The samples obtained in
\textit{Data set} 1 and \textit{Data set} 2 can be used to estimate the CDF of
acoustic noise levels at locations marked by numbers one to nine on path ${\cal
	P}_1$ as shown in Figure \ref{map} using \eqref{empirical_finite}. Since the CDF of
sound levels at a given location on path ${\cal P}_1$ is not know, a static
sound level meter was used to measure acoustic noise levels at locations marked with
numbers one to nine along path ${\cal P}_1$ shown in Figure \ref{map}. These
samples are used to compute the empirical CDF of acoustic noise levels at the
numbered locations on path ${\cal P}_1$. The CDF computed at a numbered location
along path ${\cal P}_1$ using samples from the fixed sensor is compared with the
CDF computed using \textit{Data set} 1 and \textit{Data set} 2 at that location.

\begin{figure*}
	\centering
	\subfloat[]{\includegraphics[scale=0.55]{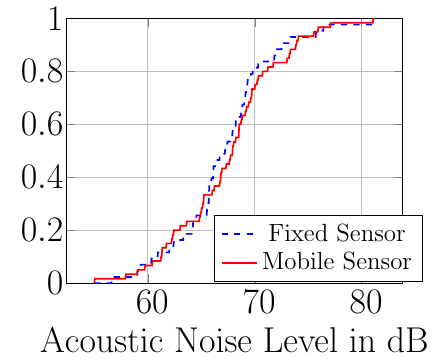}%
		\label{fig:cyc_wal_1}}
	\hfil
	\subfloat[]{\includegraphics[scale=0.55]{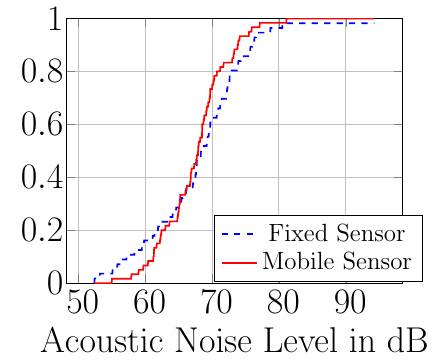}%
		\label{fig:cyc_wal_2}}
	\caption{Comparison of empirical CDF of sound
		levels obtained by experimentation at location 9 along path ${\cal P}_1$ in
		Figure \ref{map}, for two different sampling rates of mobile sensor:
		(\textit{a}) fixed sensor  versus mobile sensor for \textit{Data set} 1 (Higher
		spatial sampling rate $n_1$) (\textit{b}) fixed sensor versus mobile sensor for
		\textit{Data set} 2 (Lower spatial sampling rate $n_2$).}
	\label{cyc_wal}
\end{figure*}
The range, median and quartiles of acoustic noise levels along path ${\cal P}_1$
can be visualised using a box plot illustrated in Figure \ref{boxplot}. The
variation in average noise level is about 20 dB (ratio of 100) while the
variation in dynamic range exceeds 30dB (ratio of 1000). Figure
\ref{scatterplot} shows the variation of acoustic noise levels during each trial of the
mobile sensing experiment in \textit{Data set 1}  

In Figure \ref{cyc_wal}, the distribution learning method is compared for two
different spatial sampling rates $n_1$ and $n_2$. The maximum absolute pointwise error in
learning the empirical CDFs in Figure \ref{fig:cyc_wal_1} is less than the
maximum absolute pointwise error between the empirical CDFs in Figure
\ref{fig:cyc_wal_2}. This shows that the error in learning the empirical CDFs is
lower for \textit{Data set} 1 having higher spatial sampling rate compared to
\textit{Data set} 2. This validates  our theoritical claims. Through experiments we have verified the impact of spatial sampling rate on distribution learning. However,  due to limited amount of data, the effect of increasing the number of trials is not investigated in the experiment.

\section{Conclusion}

We have proposed a data-driven method for learning the statistical distribution
of a time-varying spatial field as a function of spatial location under some
simple assumptions. A bounded and Lipschitz continuous time varying spatial field
is sampled using a location-unaware mobile sensor. The unknown sampling locations
are modeled by an unknown renewal process. We have shown that the error in
learning the distribution of field value as a function of space decreases with increasing
spatial sampling rate $n$ of the location-unaware mobile sensor and number of trials of the mobile sensing experiment $N$. 
(i) We have a bound of $\mathcal{O}\left(\frac{1}{\sqrt{n}}\right)$ on the maximum error in estimating the CDF of the spatial field, as a function of spatial location. This is an improvement over the bound of $\mathcal{O}\left(\frac{1}{n^{1/3}}\right)$ given in \cite{NIPS2019}. This bound holds as $N$ goes to infinity. 
(ii) In addition, we show that with probability atleast $1 - \delta$, the error in estimating the CDF of the spatial field as a function of spatial location is of $\mathcal{O}\left(\frac{1}{\sqrt n} + \frac{1}{\sqrt N}\right)$ for a finite $N$. We validate this
claim using simulations. We have created a data set with 43 trials of the mobile
sensing experiment measuring acoustic noise levels along a fixed length path
using a location-unaware mobile sound level meter. The experimental results also
validate the distribution learning method.

\section*{Acknowledgements}
The author is grateful for the guidance and support received from Animesh Kumar\footnote{email: animekum@outlook.com} in this work.

\bibliographystyle{unsrt}
\bibliography{References}

	\appendix
    \section
	\textit{Proof of Lemma} 1: \label{proof_lemma_1}
	The field is assumed to be Lipschitz continuous, so 
	\begin{align} \label{Lipschitz}
		| X(s)-\hat{X}(s) | \leq \alpha \left| S_{\lfloor (M-1)s\rfloor +1} - s \right|
		= \alpha D,
	\end{align}
	where $D=\left| S_{\lfloor (M-1)s\rfloor + 1} - s \right|$. From
	\eqref{Lipschitz} we can write, 
	\begin{equation*}
		X(s) - \alpha D \leq \hat{X}(s).
	\end{equation*}
	So,
	\begin{align*}
		 \pP( \hat{X}(s) \leq x, x < X(s) \leq x + \varepsilon) 
		&  \leq \pP( X(s) - \alpha D \leq x, x < X(s) \leq x + \varepsilon) \nonumber \\
		& = \pP(x < X(s) \leq x + \min(\alpha D, \varepsilon) ).
	\end{align*}
	Let the random variable $ Z = \min(\alpha D, \varepsilon)$. The conditional probability that $X(s)$ lies between $x$ and $x+Z$, given $Z=z$, is 
	\begin{align}
		\pP(x < X(s) \leq x + Z | Z=z ) \leq z\max_{x}( f_{X(s)}( x ) ). \label{upperbound1_lemma1} 
	\end{align} 
	Probability that $X(s)$ lies between $x$ and $x+Z$ is given by 
	\begin{align}
		\pP(x < X(s) \leq x + Z ) = \eE[\pP(x < X(s) \leq x + Z | Z)], \label{equality1_lemma1}
	\end{align}
	where expectation is taken over $Z$. From \eqref{upperbound1_lemma1} and \eqref{equality1_lemma1} we get
	\begin{align}
		\pP(x < X(s) \leq x + Z ) \leq \eE[Z]\max_{x}( f_{X(s)}( x ) ). \label{mainbound_lemma1}
	\end{align}
	The random variable $Z = \min(\alpha D, \varepsilon)$ can be written as
	\begin{equation*}
		Z = \varepsilon .\mathbb{I}(\alpha D > \varepsilon) + \alpha D.
		\mathbb{I}(\alpha D \leq  \varepsilon),
	\end{equation*}
	where $\mathbb{I}(.)$ denotes the indicator function. The conditional expectation of $Z$ given $D$ is 
	\begin{equation*}
		\eE[Z|D=d] = \varepsilon .\mathbb{I}(\alpha d > \varepsilon) + \alpha d.
		\mathbb{I}(\alpha d \leq  \varepsilon).
	\end{equation*}
	So, 
	\begin{align*}
		\eE[Z] & =\varepsilon.\underbrace{\pP(\alpha D > \varepsilon )}_{*} + \eE[\alpha D].\pP(\alpha D \leq
		\varepsilon).
	\end{align*}
	By applying Markov inequality on (*) we get
	\begin{align}
		\eE[Z] & \leq \eE[\alpha D] + \eE[\alpha D].\pP(\alpha D \leq \varepsilon)  \nonumber  \\
		& = \eE[\alpha D](1 + \pP(\alpha D \leq \varepsilon) ) \nonumber \\ 
		& \leq 2\eE[\alpha D]. \label{inequality2_lemma1}
	\end{align}
	By Jensen's inequality, $(\eE[\alpha D])^2 \leq \eE[\alpha^2 D^2]$. From Proposition \ref{prop_1} we have 
	\begin{equation*}
		\eE[\alpha^2 D^2] \leq \alpha^2((n+\lambda-1)s(1-s)+C)\frac{\lambda^2}{n^2}.
	\end{equation*}
	Therefore, 
	\begin{equation*}
		\eE[\alpha D] \leq \alpha\frac{\lambda}{n}\sqrt{((n+\lambda-1)s(1-s)+C)}.
	\end{equation*}
	By substituting the above upper bound in \eqref{inequality2_lemma1} we get 
	\begin{align*}
		\eE[Z] \leq 2\alpha\frac{\lambda}{n}\sqrt{((n+\lambda-1)s(1-s)+C)}.
	\end{align*}
	By substituting the upper bound on $\eE[Z]$ in \eqref{mainbound_lemma1} we get
	\begin{align*} 
		 \pP( \hat{X}(s) \leq x, x < X(s) \leq x + \varepsilon)  
		 \leq 2\alpha\max_{x}( f_{X(s)}( x
		))\frac{\lambda}{n}\sqrt{((n+\lambda-1)s(1-s)+C)}.
	\end{align*}

\end{document}